\documentclass[12pt]{article}
\usepackage[centertags]{amsmath}
\usepackage{amsmath, amssymb, makeidx}
\usepackage{amsthm}
\usepackage{graphicx}
\usepackage{mathrsfs}
\usepackage{amsfonts}
\input{amssym.def}
\makeindex
\numberwithin{equation}{section}

\newtheorem{Definition}{Definition}[section]
\newtheorem{Theorem}[Definition]{Theorem}

\newtheorem{Corollary}[Definition]{Corollary}

\begin{document}

\title{\bf b-lattice of nil-extensions of rectangular skew-rings}
\author{\textbf{ S. K. Maity and R. Chatterjee}}
\date{}

\maketitle

\vspace{-2em}
\begin{center}
{\small \it Department of Mathematics, University  of  Burdwan,}\\
{\small \it Golapbag, Burdwan-713104, India.}\\
{\small e-mail: skmaity@math.buruniv.ac.in}
\end{center}

\begin{abstract}
 {\footnotesize  \noindent Every quasi completely regular semiring
 is a b-lattice of completely Archimedean semirings, i.e., a b-lattice
 of nil-extensions of  completely simple semirings. In this paper
 we consider the semiring which is a b-lattice of nil-extensions of
 orthodox completely simple semirings.}
\end{abstract}



\section{Introduction}
Semiring is one of the many concepts of universal algebra which is
an established and  recognized area of study. The structure of
semirings has been studied by many authors, for example, by F.
Pastijn, Y. Q. Guo, M. K. Sen, K.  P. Shum and others
(\cite{past},  \cite{sms}). In \cite{mg} the authors studied
properties of quasi completely regular semirings. They proved
that a semiring is a quasi completely regular semiring if and
only if it is a b-lattice of completely Archimedean semirings. In
\cite{skmg}, the authors proved that a semiring is a completely
Archimedean semiring if and only if it is nil-extension of a
completely simple semiring. In this paper we intend to study the
semirings which are b-lattice of orthodox completely Archimedean
semirings. The preliminaries and prerequisites we need for this
article are discussed in section 2. In section 3, we discuss our
main result.

\section{Preliminaries}

A \emph{semiring} $(S,+,\cdot)$ is a type (2, 2)-algebra whose
semigroup reducts $(S,+)$ and $(S,\cdot)$ are connected by ring
like distributivity, that is, $a(b+c) = ab+ac$ and $(b+c)a =
ba+ca$ for all $a,b,c \in S$. An element $a$ in a semiring $(S,
+, \cdot)$ is said to be \emph{additively regular} if there exists
an element $x \in S$ such that $a = a+x+a$. An element $a$ of a
semiring $(S,+,\cdot)$ is called \emph{completely regular}
\cite{sms} if there exists an element $x \in S$ such that $a =
a+x+a$, $a+x = x+a$ and $a(a+x) = a+x$. We call a semiring
$(S,+,\cdot)$\emph{ completely regular} if every element of $S$
is completely regular.  A semiring $(S,+,\cdot)$ is called a
\emph{skew-ring} if its additive reduct $(S,+)$ is an arbitrary
group.  A semiring $(S,+,\cdot)$ is called a \emph{b-lattice} if
$(S, \cdot)$ is a band and $(S,+)$ is a semilattice. If both the
reducts $(S,+)$ and $(S,\cdot)$ are bands, then the semiring $(S,
+, \cdot)$ is said to be an \emph{idempotent semiring}. An
element $a$ in a semiring $(S, +, \cdot)$ is said to be
\emph{additively quasi regular} if there exists a positive
integer $n$ such that $na$ is additively regular. An element $a$
in a semiring $(S, +, \cdot)$ is said to be \emph{quasi
completely regular} \cite{mg} if there exists a positive integer
$n$ such that $na$ is completely regular. Naturally, a semiring
$(S, +, \cdot)$ is said to be quasi completely regular if all the
elements of $S$ are quasi completely regular.

A semigroup $S$ is a \emph{rectangular band} if for every $a,x
\in S$, $a=axa$. An element $x$ of a semigroup $S$ is a \emph{left
zero} of $S$ if $xa=x$ for all $a \in S$. A semigroup all of whose
elements are left zeros is said to be a \emph{left zero
semigroup}. A \emph{rectangular group} is isomorphic to a direct
product of a rectangular band and a group. Again a \emph{left
group} is isomorphic to a direct product of a left zero semigroup
and a group.  A semigroup $S$ in which the idempotents form a
subsemigroup is \emph{orthodox}. A completely simple orthodox
semigroup is a rectangular group, and conversely. We naturally
aim at extending the concepts of rectangular group and left group
to that of semirings.

Throughout this paper, we always let $E^+(S)$ and $Reg^+S$
respectively be the set of all additive idempotents and the set of
all additively regular elements of the semiring $S$. Also we
denote the set of all additive inverse elements of an additively
regular element $a$ in a semiring $(S, +, \cdot)$ by $V^+(a)$. As
usual, we denote the Green's relations on the semiring $(S, +,
\cdot)$ by $\mathscr{L}$,  $\mathscr{R}$, $\mathscr{D}$,
$\mathscr{J}$ and $\mathscr{H}$ and correspondingly, the
$\mathscr{L}$-relation, $\mathscr{R}$-relation,
$\mathscr{D}$-relation, $\mathscr{J}$-relation and
$\mathscr{H}$-relation on $(S, +)$ are denoted by
$\mathscr{L}^+$, $\mathscr{R}^+$, $\mathscr{D}^+$,
$\mathscr{J}^+$ and $\mathscr{H}^+$, respectively. In fact, the
relations $\mathscr{L}^+$, $\mathscr{R}^+$, $\mathscr{D}^+$,
$\mathscr{J}^+$ and $\mathscr{H}^+$ are all congruence relations
on the multiplicative reduct $(S, \cdot )$. Thus if any one of
these happens to be a congruence on $(S, +)$, it will be a
congruence on the semiring $(S, +, \cdot)$. We further denote the
Green's relations on a quasi completely regular semigroup as
$\mathscr{L^*}$,$\mathscr{R^*}$, $\mathscr{D^*}$,$\mathscr{J^*}$
and $\mathscr{H^*}$. For other notations and terminologies not
given in this paper, the reader is referred to the texts of Howie
~\cite{h} \& Golan~\cite{g}. Now we state some important
definitions and results on quasi completely regular semirings and
completely simple semirings.

\begin{Definition}
~\cite{sms} A completely regular semiring $(S,+,\cdot)$ is called
a completely simple semiring if $\mathscr{J}^+ = S \times S$.
\end{Definition}

\begin{Theorem}
\cite{smw} \label{rmr} Let $R$ be a skew-ring, $(I, \cdot)$ and
$(\Lambda, \cdot)$ are bands such that $I \cap \Lambda = \{ o \}$.
Let $P = (p_{_{\lambda, i}})$ be a matrix over $R$, $i \in I$,
$\lambda \in \Lambda$ and assume

1.  $p_{\lambda,o} = p_{o,i} = 0$,

2.  $p_{\lambda \mu, k j} = p_{\lambda \mu, i j} - p_{\nu \mu, i
j} + p_{\nu \mu, k j}$,

3.  $p_{\mu \lambda, jk} = p_{\mu \lambda, j i} - p_{\mu \nu, j i}
+ p_{\mu \nu, j k}$,

4.  $a p_{\lambda, i} = p_{\lambda, i} a = 0$,

5.  $ab + p_{o \mu, i o} = p_{o \mu, i o} + ab$, \label{123}

6.  $ab + p_{\lambda o, o j} = p_{\lambda o, o j} + ab$,
\label{234} \, \,  for all $i, j, k \in I$, $\lambda, \mu, \nu
\in \Lambda$ and $a,b \in R$.

\noindent Let $\mathscr{M}$ consist of all the elements of $I
\times R \times \Lambda$ and define operations $`+$' and
`$\cdot$' on $\mathscr{M}$ by

\vspace{-1.5em}
\begin{center}
$(i,a,\lambda) + (j,b,\mu) = (i, a+ p_{\lambda,j} +b, \mu)$\\
\end{center}
\vspace{-1.2em} and \vspace{-1.5em}
\begin{center}
$(i,a,\lambda) \cdot (j,b,\mu) = (ij, -p_{\lambda \mu, ij} + ab, \lambda \mu)$.
\end{center}

\noindent Then $(\mathscr{M},+,\cdot)$ is a completely simple
semiring. Conversely, every completely simple semiring is
isomorphic to such a semiring.
\end{Theorem}

\noindent The semiring constructed in the above theorem is denoted by
 $\mathscr{M} (I, R, \Lambda; P)$ and is called a \emph{Rees matrix semiring}.

\begin{Corollary}
~\cite{smw} \label{cor} Let $\mathscr{M}(I, R, \Lambda ; P)$ be a
Rees matrix semiring. Then $p_{\lambda \mu,i j} = p_{\lambda o,o
j} + p_{o \mu,i o}$ holds for all $i,j \in I$; $\lambda, \mu \in
\Lambda$. This yields $p_{\lambda,i} = p_{\lambda o, o i} + p_{o
\lambda,i o}$ and hence by assumption (5) and (6) stated in
Theorem \ref{rmr}, $ab + p_{\lambda, i} = p_{\lambda, i} + ab$
for all $i \in I$; $\lambda \in \Lambda$ and $a, b \in R$.
\end{Corollary}

\begin{Definition}
Let $(S,+,\cdot)$ be an additively quasi regular semiring. We
consider the relations $\mathscr{L}^{*^{+}}$,
$\mathscr{R}^{*^{+}}$, $\mathscr{J}^{*^{+}}$,
$\mathscr{H}^{*^{+}}$ and $\mathscr{D}^{*^{+}}$ on $S$ defined
by: for $a, b \in S$;
\begin{center}
$a \mathscr{L}^{*^{+}} b$ if and only if $ma \mathscr{L}^{+} nb$,

$a \mathscr{R}^{*^{+}} b$ if and only if $ma \mathscr{R}^{+} nb$,

$a \mathscr{J}^{*^{+}} b$ if and only if $ma \mathscr{J}^{+} nb$,

$\mathscr{H}^{*^{+}} = \mathscr{L}^{*^{+}} \cap
\mathscr{R}^{*^{+}}$ and $\mathscr{D}^{*^{+}} =
\mathscr{L}^{*^{+}} o \, \, \mathscr{R}^{*^{+}}$
\end{center}

\noindent where $m$ and $n$ are the smallest positive integers such
that $ma$ and $nb$ are additively regular.

\end{Definition}

\begin{Definition}
\cite{mg}
 A quasi completely regular semiring $(S,+,\cdot)$ is called completely Archimedean
if  $\mathscr{J}^{*^{+}} = S\times S$.
\end{Definition}

\begin{Definition}
\cite{mg} Let $R$ be subskew-ring of a semiring $S$. If for every
$a\in S$ there exists a positive integer $n$ such that $na \in
R$, then $S$ is said to be a quasi skew-ring.
\end{Definition}

\begin{Definition}
A congruence $\rho$ on a semiring $S$ is called a b-lattice
congruence (idempotent semiring congruence) if $S/\rho$ is a
b-lattice (respectively, an idempotent semiring). A semiring $S$
is called a b-lattice (idempotent semiring) $Y$ of semirings
$S_{\alpha}(\alpha \in Y)$ if $S$ admits a b-lattice congruence
(respectively, an idempotent semiring congruence)  $\rho$ on $S$
such that $Y = S/\rho$ and each $S_{\alpha}$ is a $\rho$-class.
\end{Definition}

\begin{Theorem}
~\cite{mg}
\label{qcrs}The following conditions on a semiring $(S,+,\cdot)$
are equivalent.

1.  $S$ is a quasi completely regular semiring.

2.  Every $\mathscr{H}^{*^{+}}$-class is a quasi skew-ring.

3.  $S$ is (disjoint) union of quasi skew-rings.

4.  $S$ is a b-lattice of completely Archimedean semirings.

5.  $S$ is an idempotent semiring of quasi skew-rings.
\end{Theorem}

\begin{Definition}
\cite{skmg} Let $(S, +, \cdot)$ be a semiring. A nonempty subset
$I$ of $S$ is said to be a bi-ideal of $S$ if $a\in I$ and $x\in
S$ imply that $a+x,\,\, x+a,\,\, ax,\,\, xa\, \in I$.
\end{Definition}

\begin{Definition}
\cite{skmg} Let $I$ be a bi-ideal of a semiring $S$. We define a
relation $\rho_{_I}$ on $S$ by $a \rho_{_I} b$ if and only if
either $a, b \in I$ or $a = b$ where $a, b \in S$. It is easy to
verify that $\rho_{_I}$ is a congruence on $S$. This congruence
is said to be Rees congruence on $S$ and the quotient semiring
$S/\rho_{_I}$ contains a zero, namely $I$. This quotient semiring
$S/\rho_{_I}$ is said to be the Rees quotient semiring and is
denoted by $S/I$. In this case the semiring $S$ is said to be an
ideal extension or simply an extension of  $I$ by the semiring
$S/I$. An ideal extension $S$ of  a semiring $I$ is said to be a
nil-extension of $I$ if for any $a \in S$ there exists a positive
integer $n$ such that $na \in I$.
\end{Definition}

\begin{Theorem}
~\cite{skmg} A semiring $S$ is a quasi skew-ring if and only if
$S$ is a nil-extension of a skew-ring.
\end{Theorem}

\begin{Theorem}
~\cite{skmg}
\label{cas}
The following conditions on a semiring are equivalent:

1.  $S$ is a completely Archimedean semiring.

2.  $S$ is a nil-extension of a  completely simple semiring.

3.  $S$ is  Archimedean and quasi completely regular.
\end{Theorem}

\section{Main Results}

In this section we characterize b-lattice of nil-extensions of
rectangular skew-rings.

\begin{Definition}
An idempotent semiring $(S, +, \cdot)$  is said to be a
rectangular band semiring if $(S, +)$ is a rectangular band.
\end{Definition}

\begin{Definition}
A semiring $(S, +, \cdot)$  is said to be a rectangular skew-ring
if it is a completely simple semiring and $(S, +)$ is an orthodox
semigroup, i.e., $E^+(S)$ is a subsemigroup of the semigroup $(S,
+)$ .
\end{Definition}

\begin{Theorem}
A semiring is a rectangular skew-ring if and only if it is
isomorphic to direct product of a rectangular band semiring and a
skew-ring.
\end{Theorem}

\begin{proof}
Let $(S, +, \cdot)$ be a rectangular skew-ring. Then by
definition $(S, +, \cdot)$ is a completely simple semiring and
$E^+(S)$ is a subsemigroup of the semigroup $(S, +)$. Let $S =
\mathscr{M} (I, R, \Lambda; P)$. As $E^+(S)$ is a subsemigroup of
the semigroup $(S, +)$ then we have $(i, -p_{\lambda,i}, \lambda)
+ (j, -p_{\mu,j}, \mu) = (i, -p_{\mu,i}, \mu)$. This implies $
 -p_{\lambda,i} + p_{\lambda,j} -p_{\mu,j}  =  -p_{\mu,i}$, i.e., $p_{\lambda,i} -p_{\mu,i} + p_{\mu,j}  =
 p_{\lambda,j}$. Putting $j = o$, we get $ p_{\lambda,i} =
 p_{\mu,i}$.  Similarly, putting $i=o$ we get $ p_{\lambda,j} =
 p_{\mu,j}$.

\noindent Obviously, $E^+(S)$ is a rectangular band semiring.

\noindent Now we define a mapping $\phi: S \rightarrow E^+(S)
\times R$ by
\begin{center}
$\phi(i, a, \lambda) = \Big((i, -p_{\lambda,i}, \lambda),
p_{\lambda,i} + a \Big) $ for all $(i, a, \lambda)\in S$.
\end{center}

\noindent Now,
\begin{center}
\begin{tabular}{ccl}
$\phi(i, a, \lambda) + \phi(j,b,\mu)$ & = & $\Big((i,
-p_{\lambda,i}, \lambda),  p_{\lambda,i} + a \Big) + \Big((j,
-p_{\mu,j}, \mu),  p_{\mu,j} + b \Big) $ \\

$ $ & = & $\Big((i, -p_{\lambda,i}, \lambda) + (j, -p_{\mu,j},
\mu),  p_{\lambda,i} + a + p_{\mu,j} + b \Big) $\\

$ $ & = & $ \Big((i, -p_{\mu,i}, \mu), p_{\mu,i}+ a +
p_{\lambda,j} + b \Big) $\\

$ $ & = & $ \phi(i,  a + p_{\lambda,j} + b , \mu)$\\

$ $ & = & $\phi \Big((i, a, \lambda) + (j,b,\mu)\Big)$.
\end{tabular}
\end{center}

\noindent Again,
\begin{center}
\begin{tabular}{ccl}
$\phi(i, a, \lambda) \cdot \phi(j,b,\mu)$ & = & $\Big((i,
-p_{\lambda,i}, \lambda),  p_{\lambda,i} + a \Big) \cdot \Big((j,
-p_{\mu,j}, \mu),  p_{\mu,j} + b \Big) $\\

$ $ & = & $\Big((i, -p_{\lambda,i}, \lambda)(j, -p_{\mu,j}, \mu),
(p_{\lambda,i} + a) (p_{\mu,j} + b\Big) $\\

$ $ & = & $\Big((ij, -p_{\lambda \mu, ij}, \lambda \mu), ab
\Big)$.
\end{tabular}
\end{center}

\noindent Also
\begin{center}
\begin{tabular}{ccl}
$\phi \Big((i, a, \lambda) \cdot (j,b,\mu)\Big)$ & = & $\phi(ij,
-p_{\lambda \mu, ij}+ab, \lambda \mu)$\\

$ $ & = & $\Big((ij, -p_{\lambda \mu, ij}, \lambda \mu),
p_{\lambda \mu, ij} -p_{\lambda \mu, ij} + ab \Big)$\\

$ $ & = & $\Big((ij, -p_{\lambda \mu, ij}, \lambda \mu), ab
\Big)$.
\end{tabular}
\end{center}

So
$\phi(i, a, \lambda) \cdot \phi(j,b,\mu) = \phi \Big((i, a, \lambda) \cdot (j,b,\mu)\Big)$.

\noindent Hence $\phi$ is a homomorphism. It is easy to see that
$\phi$  is one-one and onto. Hence $\phi$ is an isomorphism from
a rectangular skew-ring $S$ onto a direct product of a
rectangular band semiring $E^+(S)$ and a skew-ring $R$.

Converse part is obvious. \end{proof}

\begin{Theorem}
\label{rskew} The following conditions on a semiring $S$ are
equivalent:

(i) $S$ is a b-lattice of nil-extensions of rectangular
skew-rings,

(ii) $S$ is a quasi completely regular semiring and for every
$e,f \in E^+(S)$, there exists $n \in \mathbb{N}$ such that
$n(e+f) = (n+1)(e+f)$,

(iii) $S$ is additively quasi regular, $b^2 \, \mathscr{H}^{*^{+}}
\, b$ for all $b \in S$ and $a=a+x+a$ implies $a=a+2x+2a$.
\end{Theorem}

\begin{proof}
(i) $\Longrightarrow$ (ii):\, Let $S$ be a b-lattice of
nil-extensions of rectangular skew-rings. Then $S$ is a b-lattice
of nil-extensions of orthodox completely simple semirings. So by
Theorem 2.12, it follows that $S$ is a b-lattice of completely
Archimedean semirings and hence by Theorem 2.8, we have $S$ is a
quasi completely regular semiring. Again, $(S, +)$ is a
semilattice of nil-extensions of orthodox completely simple
semigroups, i.e., $(S, +)$ is a semilattice of nil-extensions of
rectangular groups. Then by Theorem X.2.1 ~\cite{bs}, for every
$e,f \in E^+(S)$, there exists $n \in Z^{+}$ such that $n(e+f) =
(n+1)(e+f)$.

(ii) $\Longrightarrow$ (iii): \, Let $(S, +, \cdot)$ be a quasi
completely regular semiring and for every $e, f \in E^+(S)$, there
exists $n \in Z^{+}$ such that $n(e+f) = (n+1)(e+f)$. Hence $S$
is an additively quasi regular semiring and $b^2 \,
\mathscr{H}^{*^{+}} \, b$ for all $b \in S$. Let $a=a+x+a$. Then
$a$ is additively regular and hence by Theorem X.2.1 ~\cite{bs},
it follows that $a = a+2x+2a$.

(iii) $\Longrightarrow$ (i): \, Let $S$ is additively quasi
regular, $b^2 \, \mathscr{H}^{*^{+}} \, b$ for all $b \in S$ and
$a=a+x+a$ implies $a=a+2x+2a$. Then by Theorem X.2.1~\cite{bs},
$(S,+)$ is a semilattice of nil-extensions of rectangular groups.
This implies $(S,+)$ is a GV-semigroup. To complete the proof, it
suffices to show that every additively regular element of $S$ is
completely regular. For this let $a = a+x+a$. Then $a$ is regular
and hence $a$ is completely regular in the semigroup $(S, +)$. So
there exists an element $y \in S$ such that $a = a+y+a$ and $a+y =
y+a$. Now since $\mathscr{H}^{*^{+}}$ is a congruence on $(S,
\cdot)$ and $(a+y) \, \mathscr{H}^{*^{+}} \, a$, it follows that
$a(a+y) \, \mathscr{H}^{*^{+}} \, a^2 \, \mathscr{H}^{*^{+}} \, a
\, \mathscr{H}^{*^{+}} \, (a+y)$. Since each
$\mathscr{H}^{*^{+}}$-class contains a unique additive idempotent
and $a(a+y) \, \mathscr{H}^{*^{+}} \, (a+y)$, it follows that
$a(a+y) = a+y$. Hence $S$ is a quasi completely regular semiring.
Consequently, $S$ is a b-lattice of nil-extensions of rectangular
skew-rings.
\end{proof}

\begin{Theorem}
A semiring $S$ is nil-extension of a rectangular skew-ring if and
only if $S$ is a completely Archimedean semiring and $E^+(S)$ is
a subsemigroup of $(S,+)$.
\end{Theorem}

\begin{proof}
Let $(S, +, \cdot)$  be a nil-extension of a rectangular skew-ring
$(K, +, \cdot)$. Then  $(K, +, \cdot)$ is a completely simple
semiring  and $E^+(K)$ is a subsemigroup of $(K,+)$. Hence $S$ is
a completely Archimedean semiring. Clearly, $E^+(S) = E^+(K)$ and
thus $E^+(S)$ is a subsemigroup of $(S,+)$.

Conversely, let $(S, +, \cdot)$ be a completely Archimedean
semiring and $E^+(S)$ be a subsemigroup of $(S,+)$. Then $(S, +,
\cdot)$ is  a nil-extension of a completely simple semiring  $(K,
+, \cdot)$. Also, $E^+(S) = E^+(K)$ implies $(K, +, \cdot)$ is a
completely simple semiring such that $(K, +)$ is orthodox. Thus
$K$ is a rectangular skew-ring and hence $S$ is nil-extension of
a rectangular skew-ring $K$.
\end{proof}

\begin{Theorem}
The following conditions  are equivalent on a semiring $(S, +, \cdot)$ :

(i) $(S, +, \cdot)$ is a quasi completely regular semiring and
$E^+(S)$ is a subsemigroup of $(S,+)$.

(ii) $S$ is additively quasi regular, $b^2 \, \mathscr{H}^{*^{+}}
\, b$ for all $b \in S$, $a=a+x+a$ implies $a=a+2x+2a$ and
$Reg^+S$ is a subsemigroup of $(S,+)$.

(iii) $S$ is a b-lattice of nil-extensions of rectangular
skew-rings and $E^+(S)$ is a subsemigroup of $(S,+)$.
\end{Theorem}

\begin{proof}
(i) $\Longrightarrow$ (ii):\, Let $S$ is a quasi completely
regular semiring and $E^+(S)$ is a subsemigroup of $(S,+)$. Then
clearly $S$ is a b-lattice of nil-extensions of rectangular
skew-rings. Hence by Theorem ~\ref{rskew}, it follows that $S$ is
additively quasi regular, $b^2 \, \mathscr{H}^{*^{+}} \, b$ for
all $b \in S$ and $a=a+x+a$ implies $a=a+2x+2a$. Also $E^+(S)$ is
a subsemigroup of $(S,+)$. Then by proposition X.2.1~\cite{bs},
for any $a,b,x,y \in S$, $a=a+x+a$ and $b=b+y+b$ implies
$a+b=a+b+y+x+a+b$, i.e., $a, b \in Reg^+S$ implies $a+b \in
Reg^+S$. Hence $Reg^+S$ is a subsemigroup of $(S,+)$.

(ii) $\Longrightarrow$ (iii):\, Let $S$ be additively quasi
regular, $b^2 \, \mathscr{H}^{*^{+}} \, b$ for all $b \in S$,
$a=a+x+a$ implies $a=a+2x+2a$ and $Reg^+S$ is a subsemigroup of
$(S,+)$. Then by Theorem ~\ref{rskew},  $S$ is a b-lattice of
nil-extensions of rectangular skew-rings. Let $e, f \in E^+(S)
\subset Reg^+(S)$. Then $e+f \in Reg^+(S)$ and let $x \in S$ be an
additive inverse of $e+f$. Then $f+x+e = f+(x+e+f+x)+e =
2(f+x+e)$. Now $(e+f)+(f+x+e)+(e+f) = e+f+x+e+f = e+f$ implies
$e+f = (e+f)+2(f+x+e)+2(e+f) = (e+f)+(f+x+e)+(e+f)+(e+f) =
(e+f)+(e+f) = 2(e+f)$. Hence we have $E^+(S)$ is a subsemigroup
of $(S,+)$.

(iii) $\Longrightarrow$ (i):\, It is obvious by Theorem
~\ref{rskew}. \end{proof}

\begin{Theorem}
The following conditions  are equivalent on a semiring $(S, +, \cdot)$ :

(i) $S$ is a quasi completely regular semiring and for every $e,f
\in E^+(S)$, $e+f = f+e$,

(ii) $S$ is b-lattice of quasi skew-rings and for every $e,f \in
E^+(S)$, $e+f = f+e$,

(iii) $S$ is additively quasi regular and $Reg^+S$ is a
subsemigroup of $(S,+)$ which is a b-lattice of skew-rings.
\end{Theorem}

\begin{proof}
(i) $\Longrightarrow$ (ii): \, Let $S$ be a quasi completely
regular semiring and for every $e, f \in E^+(S)$, $e+f = f+e$.
Then $\mathscr{H}^{*+}$ is a b-lattice congruence on $S$. Hence
$S$ is b-lattice of quasi skew-rings and for every $e, f \in
E^+(S)$, $e+f = f+e$.

(ii) $\Longrightarrow$ (i): \, This part is obvious.

(i) $\Longrightarrow$ (iii): \,  Let $S$ be a quasi completely
regular semiring such that for every $e, f \in E^+(S)$, $e+f =
f+e$. Then obviously $S$ is additively quasi regular. Now let $a,b
\in Reg^+S$. Then there exist $x, y \in S$ such that $a=a+x+a$
and $b=b+y+b$. So $a+x, x+a, b+y, y+b \in E^+(S)$. Now, $a+b =
a+x+a+b+y+b = a+b+y+x+a+b$ implies $a+b \in Reg^+S$. Hence $Reg^+
S$ is a subsemigroup of $(S,+)$. Since in a quasi completely
regular semiring, every additively regular element is completely
regular, it follows that $Reg^+S$ is a completely regular
semiring. Thus by Theorem 3.6 ~\cite{sms}, $Reg^+S$ can be
regarded as a b-lattice $Y$ of completely simple semirings
$R_{\alpha}$ ($\alpha \in Y$), where $Y = S/\mathscr{J}^+$ and
each $R_{\alpha}$ is a $\mathscr{J}^+$-class in $S$. Let $e$ and
$f$ be two additive idempotents in $R_{\alpha}$. Then $e+f = f+e$
and $e \, \mathscr{J}^+ \, f$. Since $R_{\alpha}$ is a completely
simple semiring then $(R_{\alpha}, +)$ is a completely simple
semigroup and so $e \, \mathscr{D}^+ \, f$ and thus we have
$e=f$. This shows that each $R_{\alpha}$ contains a single
additive idempotent, so that $(R_{\alpha}, +)$ is a group and
hence $(R_{\alpha}, +, \cdot)$ is a skew-ring. In other words, we
have shown that $Reg^+S$ is a b-lattice $Y$ of skew-rings
$R_{\alpha}$.

(iii) $\Longrightarrow$ (i): \, Let $S$ be an additively quasi
regular semiring and $Reg^+S$ be a subsemigroup of $(S,+)$ which
is a b-lattice of skew-rings. Then clearly $S$ is a quasi
completely regular semiring. Moreover, $(Reg^+(S), +)$ is a
Clifford semigroup. Let $e, f \in E^+(S)$. Then $e$ and $f$ are
two idempotents in the Clifford semigroup $(Reg^+(S), +)$. Hence
$e+f = f+e$.
\end{proof}

\begin{Definition}
An idempotent semiring $(S, +, \cdot)$  is said to be a left zero
semiring if $(S, +)$ is a left zero band.
\end{Definition}

\begin{Definition}
A semiring $S$ is said to be a left skew-ring if it is isomorphic
to a direct product of a left zero semiring and a skew-ring.
\end{Definition}

\noindent We recall that a left group is isomorphic to a direct
product of a left zero semigroup and a group. Thus, if a semiring
$(S, +, \cdot)$  is a left skew-ring then the semiring $(S, +,
\cdot)$ is a rectangular skew-ring and the semigroup $(S, +)$ is
a left group and conversely.

\begin{Theorem}
\label{leftskew}
The following conditions  are equivalent on a semiring $(S, +, \cdot)$ :

(i) $S$ is a b-lattice of nil-extensions of left skew-rings.

(ii) $S$ is a quasi completely regular semiring and for every $e,
f \in E^+(S)$, there exists $n \in \mathbb{N}$ such that $n(e+f) =
n(e+f+e)$.

(iii) $S$ is additively quasi regular, $b^2 \,
\mathscr{H}^{*^{+}} \, b$ for all $b \in S$ and $a=a+x+a$ implies
$a+x=a+2x+a$.
\end{Theorem}

\begin{proof}
(i) $\Longrightarrow$ (ii):\, Let $S$ be a b-lattice $Y$ of
nil-extensions of left skew-rings $S_{\alpha}$ ($\alpha \in Y$).
Then  $S$ is a b-lattice of nil-extensions of rectangular
skew-rings. Hence by Theorem~\ref{rskew}, $S$ is a quasi
completely regular semiring. Again $(S, +)$ is a semilattice of
nil-extensions of left groups $(S_{\alpha}, +)$ ($\alpha \in Y$).
Hence by Theorem X.2.2 ~\cite{bs}, it follows that for every $e,
f \in E^+(S)$, there exists a positive integer $n$ such that
$n(e+f) = n(e+f+e)$.

(ii) $\Longrightarrow$ (iii): \, Let $S$ be a quasi completely
regular semiring and for every $e, f \in E^+(S)$, there exists $n
\in \mathbb{N}$ such that $n(e+f) = n(e+f+e)$. Then clearly $S$ is
additively quasi regular and $b^2 \, \mathscr{H}^{*^{+}} \, b$
for all $b \in S$. Again $(S, +)$ is a GV-semigroup such that for
every $e, f \in E^+(S)$, there exists $n \in \mathbb{N}$ such that
$n(e+f) = n(e+f+e)$. Hence by Theorem X.2.2 ~\cite{bs}, follows
that $a=a+x+a$ implies  $a+x = a+2x+a$.

(iii) $\Longrightarrow$ (i): \, Let $S$ be an additively quasi
regular such that $b^2 \, \mathscr{H}^{*^{+}} \, b$ for all $b \in
S$ and $a=a+x+a$ implies $a+x=a+2x+a$. Then $a=a+x+a = (a+x)+a =
(a+2x+a)+a = a+2x+2a$. So by Theorem~\ref{rskew}, it follows that
$S$ is a b-lattice $Y$ of nil-extensions of rectangular
skew-rings $S_{\alpha}$ ($\alpha \in Y$). By Theorem X.2.2
~\cite{bs}, it follows that $(S_{\alpha}, +)$ is a left group.
Hence $S_{\alpha}$ is a left skew-ring. Consequently, $S$ is a
b-lattice of nil-extensions of left skew-rings.
\end{proof}

\begin{Theorem}
The following conditions  are equivalent on a semiring $(S, +, \cdot)$ :

(i) $S$ is a b-lattice of nil-extensions of left skew-rings and
$E^+(S)$ is a subsemigroup of $(S,+)$.

(ii) $S$ is a quasi completely regular semiring and for every
$e,f \in E^+(S)$, $e+f = e+f+e$.
\end{Theorem}

\begin{proof}
(i) $\Longrightarrow$ (ii):\, Let $S$ be a b-lattice of
nil-extensions of left skew-rings and $E^+(S)$ is a subsemigroup
of $(S, +)$. Then by Theorem ~\ref{leftskew}, it follows that $S$
is a quasi completely regular semiring and for every $e, f \in
E^+(S)$, there exists a positive integer $n$ such that $n(e+f) =
n(e+f+e)$. Now since $E^+(S)$ is a subsemigroup of $(S, +)$, we
have $e+f, e+f+e \in E^+(S)$. Hence $e+f = n(e+f) = n(e+f+e) =
e+f+e$.

(ii) $\Longrightarrow$ (i):\, Let  $S$ be a quasi completely
regular semiring and for every $e,f \in E^+(S)$, $e+f = e+f+e$.
Then by Theorem ~\ref{leftskew}, we have $S$ is a b-lattice of
nil-extensions of left skew-rings. Also for every $e, f \in
E^+(S)$, $2(e+f) = e+f+e+f = e+f+f = e+f$ which implies that $e+f
\in E^+(S)$, i.e.,  $E^+(S)$ is a subsemigroup of $(S,+)$.
\end{proof}

\bibliographystyle{amsplain}

\end{document}